\documentclass[11pt]{amsart}

\usepackage{graphicx}
\usepackage{color}

\newtheorem{theorem}{Theorem}[section]
\newtheorem{lemma}[theorem]{Lemma}

\theoremstyle{definition}
\newtheorem{definition}[theorem]{Definition}

\newtheorem{problem}[theorem]{Problem}

\theoremstyle{remark}
\newtheorem{remark}[theorem]{Remark}

\numberwithin{equation}{section}

\begin{document}

\title[The Minkowski problem in Heisenberg groups]{The Minkowski problem in Heisenberg groups}

\author{Bin Chen}
\address{Bin Chen, Juan Zhang, Peibiao Zhao, and Xia Zhao \newline \indent }
\curraddr{}
\email{chenb121223@163.com}
\email{juanzhang892@163.com}
\email{pbzhao@njust.edu.cn}
\email{zhaoxia20161227@163.com}
\thanks{}

\author{Juan Zhang}
\address{}
\curraddr{}
\email{}

\author{Peibiao Zhao}
\address{}
\curraddr{}
\thanks{Research is supported in part by the Natural Science
Foundation of China (12271254; 12141104).}

\author{Xia Zhao}
\address{}
\thanks{Corresponding author: Peibiao Zhao}


\subjclass[2020]{53C17, 52A20}

\date{}

\dedicatory{}

\keywords{Minkowski problem, Horizontal surface area measure, Heisenberg group}

\begin{abstract}
  As we all know, the Minkowski type problem  is the cornerstone of the Brunn-Minkowski theory in Euclidean space.
  The Heisenberg group as a sub-Riemannian space is the simplest non-Abelian degenerate Riemannian space that is completely different from a Euclidean space.
  By analogy with the Minkowski type problem in Euclidean space, the  Minkowski type problem in Heisenberg groups is still open.
  In the present paper, we develop for the first time a sub-Riemannian version of Minkowski type problem in the horizontal distributions of Heisenberg groups, and  further give a positive answer to this sub-Riemannian Minkowski type problem via the variational method.
 \end{abstract}

\maketitle

\section{Introduction}
The Brunn-Minkowski theory of  convex bodies (i.e., a  compact, convex set), originated from Brunn's thesis \cite{Br} and Minkowski's paper \cite{M}, in Euclidean space $\mathbb{R}^n$ has a hot research field of convex geometry for nearly a century (the books \cite{G,S} are excellent references for this theory).
This theory centers around the study of geometric functionals of convex bodies as well as the differentials of these functionals, and depends heavily on analytic tools such as the cosine transform on the unit sphere and Monge-Amp\`{e}re type equations. The differentiations of different geometric functions are geometric measures, and these geometric measures are fundamental concepts in the Brun-Minkowski theory. It is well known that Minkowski type problem is one of the important components of the Brun-Minkowski theory. A Minkowski problem, roughly speaking, is a prescribed measure problem, which is also a characterization problem for a geometric measure generated by convex bodies.

\subsection{An overview of the Minkowski problem}

Without doubt, the most fundamental geometric functional
in the Brunn-Minkowski theory is the volume functional. Via the variation of volume functional, it produces the most important geometric measure: surface area measure.
The well-known {\it classical Minkowski problem}, which characterizes the surface area measure, asks
the necessary and sufficient conditions for a Borel measure on the unit sphere to be the surface area measure of a convex body.

More than a century ago, the classical Minkowski problem was first introduced and studied by Minkowski himself \cite{M0,M}, who demonstrated both existence and uniqueness of solutions for the problem when the given measure is either discrete or has a continuous density.
Aleksandrov \cite{A,A1} and Fenchel-Jessen \cite{FJ} independently solved the problem in 1938 for arbitrary measures.
Analytically, the Minkowski problem is equivalent to solving a  Monge-Amp\`{e}re equation. Establishing the
regularity of the solution to the Minkowski problem is difficult and has led to a long series of influential works (see, for example, Lewy \cite{L}, Nirenberg \cite{N}, Cheng-Yau \cite{C}, Pogorelov \cite{P1}, Caffarelli \cite{Ca}).

The $L_p$ Minkowski problem related to the $L_p$ surface area measure, introduced in \cite{L1}, is an extension of the classical Minkowski problem and has been intensively investigated and achieved great developments, see, e.g., \cite{BB,Ch,HLYZ,J,LO,LYZ2,St,Z} and the references therein.
The case of $p=1$ is the classical Minkowski problem, the case of $p=0$ is the logarithmic Minkowski problem (see \cite{BLYZ1}), and the case of $p=-n$ is the centro-affine Minkowski problem (see Chou-Wang \cite{CW}, Lu-Wang \cite{LW}, and Zhu \cite{Z1}).

Recently, much effort has been made to develop the nonhomogeneous problem analogous to the $L_p$ Minkowski problem; such a new problem is called the Orlicz-Minkowski problem of convex bodies (see \cite{Hab}).
It can be asked with the $L_p$ surface area measure (in $L_p$ Minkowski problem) replaced by the Orlicz surface area measure. Solutions to the Orlicz-Minkowski problem can be found in \cite{GH,Hab,HH,Li} and the references therein.

Moreover, the Minkowski-type problems (i.e., classical, $L_p$ and Orlicz Minkowski problems) also have dual analogues. It is only very recent that the dual curvature measures, which are dual to the surface area measures, were discovered in the seminal work \cite{Ha} by Huang, Lutwak, Yang and Zhang. The prescribed dual curvature measure problem is called the dual Minkowski problem.
Subsequently, Luwwak et al  \cite{LYZ5} and Zhu et al \cite{ZZ} introduced and studied dual $L_p$ Minkowski problem and dual Orlicz problem respectively.
See e.g., \cite{BH,BLYZZ,CL,HP,Ha,LLL,LSW,Zha1,Zha2} for more works in the rapidly developing dual Minkowski type problem.

With the development of Brunn-Minkowski theory in Euclidean setting, Li and Xu in a recent seminal article \cite{LX} started the study of Brunn-Minkowski theory in hyperbolic space, and solved the Horospherical
Minkowski problem and the more general Horospherical Christoffel-Minkowski problem based on the concept
of hyperbolic $p$-sum.

Motivated by the foregoing celebrated works, we  in the present paper will investigate and confirm the sub-Riemannian version of the Minkowski problem in sub-Riemannian Heisenberg groups.

\subsection{The sub-Riemannian Heisenberg group}

The Heisenberg group $\mathbb{H}^n$, the simplest non-Abelian sub-Riemannian space distinct from a Euclidean space, is the Lie group ($\mathbb{R}^{2n+1},\ast$), where we consider in $\mathbb{R}^{2n+1}\equiv\mathbb{R}^n\times
\mathbb{R}^n\times\mathbb{R}$ its usual differentiable structure and the noncommutative group law
\begin{align}\label{1.1}
g_1\ast g_2=
\bigg(z_1+z_2,t_1+t_2+\frac{1}{2}Imz_1\bar{z}_2\bigg),
\end{align}
for $g_1=(z_1,t_1)$ and $g_2=(z_2,t_2)\in\mathbb{H}^n$,
where $z=x_l+iy_l\in\mathbb{R}^{2n}$ and $t\in\mathbb{R}$.
The identity element is $e=(0,0)$,
the inverse element of $g=(z,t)$ is $g^{-1}=(-z,-t)$,
and the center of the group is $c=\{(z,t)\in\mathbb{H}^n: z=0\}$. Left translation by $g\in\mathbb{H}^n$ is the mapping $L_g: L_g(\tilde{g})=g\ast\tilde{g}$.
For any $\lambda>0$, the mapping $\delta_\lambda: \delta_\lambda(z,t)=(\lambda z,\lambda^2t)$
is called dilation.

The Lie algebra $\mathfrak{h}$ is spanned by the left-invatiant vector fields $\{X_l, Y_l, T\}$, $l=1,\cdots,n$ given by
\begin{align}\label{1.2}
\nonumber &X_l=dL_g\frac{\partial}{\partial x_l}=\frac{\partial}{\partial x_l}-\frac{y_l}{2}\frac{\partial}{\partial t},\\ &Y_l=dL_g\frac{\partial}{\partial y_l}=\frac{\partial}{\partial y_l}
+\frac{x_l}{2}\frac{\partial}{\partial t},\\
\nonumber &T=dL_g\frac{\partial}{\partial t}=\frac{\partial}{\partial t},
\end{align}
where $dL_g$ is the differential of the left-translation (\ref{1.1}). The only non-trivial bracket relations are $[X_l,Y_l]=-T$, $l=1,\cdots,n$.
In other words, any left invariant vector is a linear combination with real coefficients of the vector fields (\ref{1.2}). The Lie algebra $\mathfrak{h}$ admits a stratification. It decomposes  as the vector space direct sum
$$\mathfrak{h}=\mathfrak{h}_1\oplus\mathfrak{h}_2,$$
where $\mathfrak{h}_1$ is the subspace generated by $\{X_l,Y_l\}_{l=1}^n$ and $\mathfrak{h}_2=[\mathfrak{h}_1,\mathfrak{h}_1]$ is the one-dimensional space generated $T$. For this reason $\mathbb{H}^n$ is called a Carnot group of step 2.

The Haar measure on the Heisenberg group $\mathbb{H}^n$ is induced by the exponential mapping from the Lebesgue measure on $\mathfrak{h}$.
The natural volume in $\mathbb{H}^n$ is the Haar measure, which, up to a positive factor, coincides with Lebesgue measure in $\mathbb{R}^{2n+1}$. Lebesgue measure is also the Riemannian volume of the left-invariant metric for which $X_l,Y_l$ and $T$ are orthonormal. We denote by $Vol(E)$ the volume of a (Lebesgue) measurable set $E\subset\mathbb{H}^n$.

We denote by $\mathbb{H}_e$ the set of horizontal vectors of the form $h=(z,0)$. Fixed a point $g_0=(z_0,t_0)\in\mathbb{H}^n$,
the horizontal distribution $\mathbb{H}_{g_0}=g_0\ast\mathbb{H}_e=$
span$\{X_l(g_0),Y_l(g_0)\}_{l=1}^n$ is characterized by
\begin{align}\label{1.3}
\mathbb{H}_{g_0}=Ker\bigg(dt_0-\frac{1}{2}\langle z_0,Jdz_0\rangle_{\mathbb{R}^{2n}}\bigg),
\end{align}
where $J=( \ ^0_{-1}\ ^1_0)$.
From the left translation mapping we can set $g_0$ as the origin of $\mathbb{H}_{g_0}$.
Fixed $g_0\in\mathbb{H}^n$, any vector $v\in\mathbb{H}_{g_0}$ can be written as $v=(\bar{v},v_t)\in\mathbb{R}^{2n}\times\mathbb{R}$.
Let $\langle\cdot,\cdot\rangle_{\mathbb{H}^{n}}$ be a left invariant inner product on $\{X_l,Y_l\}_{l=1}^n$.
Then, the sub-Riemannian metric (or Kor\'{a}nyi gauge) is denoted as below (see \cite{FP})
\begin{align}\label{1.4}
\langle v,\omega\rangle_{\mathbb{H}^{n}}=
\langle\bar{v},\bar{\omega}\rangle_{\mathbb{R}^{2n}},\ \
\forall v,\omega\in\mathbb{H}_{g_0},\ g_0\in\mathbb{H}^{n},
\end{align}
where $\langle\cdot,\cdot\rangle_{\mathbb{R}^{2n}}$ is the standard scalar product on $\mathbb{R}^{2n}$.

A Lipschitz curve $\gamma(t): [0,1]\rightarrow\mathbb{H}^n$ is horizontal if $\gamma^\prime\in(\gamma(t))\ast\mathbb{H}_e$ for a.e. $t\in[0,1]$. Equivalently, $\gamma$ is horizontal if there exist functions $f_l\in L^\infty([0,1])$, $l=1,\cdots,2n$, such that
$$\gamma^\prime(t)=\sum_{l=1}^nf_lX_l(\gamma)
+f_{n+l}Y_l(\gamma),\ a.e.\ on\ [0,1].$$
The coefficients $f_l$ are unique, and by the structure of the vector fields $X_l$ and $Y_l$ they satisfy $f_l=\gamma^\prime_l$,  where $\gamma=(\gamma_1,\cdots,\gamma_{2n+1})$ are the coordinates of $\gamma$ given by the identification $\mathbb{H}^n=\mathbb{R}^{2n+1}$.

Given two points $g$ and $\tilde{g}$ we consider the set of all possible horizontal curves joining these two points
$$\Gamma(g,\tilde{g})=\{\gamma\ horizontal\ curve:\ \gamma(0)=g,\ \gamma(1)=\tilde{g}\}.$$
This set is never empty by Chow theorem (see \cite{Be}). The Carnot-Caratheodory (abbreviated CC) distance is then defined as the infimum of the length of horizontal curves of the set $\Gamma$:
$$d_{CC}(g,\tilde{g})=\inf_{\Gamma(g,\tilde{g})}
\int_0^1|\gamma^\prime(t)|dt.$$
The curve $\gamma$ is called a geodesic or length minimizing curve joining $g$ and $\tilde{g}$.
The CC ball of radius $r>0$ centered at a point $g\in\mathbb{H}^n$ is defined by
$$B_r(g)=\{\tilde{g}\in\mathbb{H}^n: d_{CC}(g,\tilde{g})<r\}.$$
We also let $B_r=B_r(0)$.
The CC gauge is given by
$$|g|_{CC}=d_{CC}(e,g).$$
The CC distance, being constructed in terms of left invariant vector fields, is
left invariant and dilation invariant on $\mathbb{H}^n$.  Namely, for any $p,g,\tilde{g}\in\mathbb{H}^n$ and $\lambda>0$ there holds
$$d_{CC}(p\ast g,p\ast\tilde{g})=d_{CC}(g,\tilde{g});\ \ d_{CC}(\delta_\lambda(g),\delta_\lambda(\tilde{g}))=
d_{CC}(g,\tilde{g}).$$

An equivalent distance on $\mathbb{H}^n$ is defined by the
so-called Kor\'{a}nyi distance (see Section 2.2 of  \cite{Cap})
$$d_{\mathbb{H}}(g,\tilde{g})=\|\tilde{g}^{-1}\ast g\|_{\mathbb{H}},$$
and Kor\'{a}nyi gauge
$$\|g\|_{\mathbb{H}}^4=(x_l^2+y_l^2)^2+16t^2.$$
Clearly, $d_{\mathbb{H}}$ is homogeneous of order 1 with
respect to the dilations $(\delta_s): \|\delta_sg\|_{\mathbb{H}}=s\|g\|_{\mathbb{H}}$. Consequently, there exist constants $C_1, C_2>0$ so that
$$C_1\|g\|_{\mathbb{H}}\leq d_{CC}(g,0)\leq C_2\|g\|_{\mathbb{H}}$$
for any $g\in\mathbb{H}^n$.
This follows immediately from compactness of the Kor\'{a}nyi unit sphere
$$\mathcal{S}^{2n}=\{g\in\mathbb{H}^n: \|g\|_{\mathbb{H}}=1\}$$
and continuity of $g\mapsto d_{CC}(g,0)$.

Now we introduce the concept of convexity in horizontal distributions (see \cite{D} for detail).
A subset $\Omega\subset\mathbb{H}^n$ is called horizontal convex (abbreviated $H$-convex) if and only if the twisted convex combination
$g\ast\delta_\lambda(g^{-1}\ast\tilde{g})\in\Omega$ for any $g\in\Omega$ and every $\tilde{g}\in\Omega\cap\mathbb{H}_g$.
Since to every notion of convexity on  sets there is a naturally associated notion of convexity for functions, from the definition above we have:
Let $\Omega$ be a $H$-convex subset of $\mathbb{H}^n$. A function $u: \Omega\rightarrow\mathbb{R}$ is called $H$-convex if, for any $g, \tilde{g}\in\Omega$, one has
$$u(g\ast\delta_\lambda(g^{-1}\ast\tilde{g}))
\leq(1-\lambda)u(g)+\lambda u(\tilde{g}),\ \ \forall\lambda\in[0,1].$$
For more detailed research on convex functions and convex sets, one can refer to \cite{Cal,GT,GM,LuM,Mo1,SY,SY1} and therein for details.

\subsection{Definitions and main results}
Due to the existence of left translation and dilation  mapping families, the geometric structures in Heisenberg groups are fundamentally different from Euclidean spaces, so the study of Brunn-Minkowski theory in Heisenberg group $\mathbb{H}^n$ is still blank, which leads us to develop Brunn-Minkowski theory in $\mathbb{H}^n$.

In the present paper, we construct the horizontal support functions and horizontal geometric measure and other basic
convex geometric elements, and develop a sub-Riemannian version of Minkowski type problem in the horizontal distributions of $\mathbb{H}^n$.

Fix a point $g_0\in\mathbb{H}^n$,
let $\mathbb{H}_{g_0}$ be the horizontal distribution through $g_0$, and a compact $H$-convex subset with non-empty interior be a $H$-convex body.
For the horizontal vector field $V=\{X_l,Y_l\}_{l=1}^n$ in $\mathbb{H}_{g_0}$ and  the horizontal geodesic
$$\gamma: [0,1]\rightarrow\mathbb{H}_{g_0}, \ \
\gamma(t)=exp(-tY_l)exp(-tX_l)exp(tY_l)exp(tX_l)(g),$$
where $exp(tV)(g)$ is the flow of horizontal vector field $V$ at time $t$ starting from $g\in\mathbb{H}_{g_0}$.

For fixed $g_0\in\mathbb{H}^n$, let $\Omega\subset\mathbb{H}_{g_0}$ be a $H$-convex body.
We denote by $\overrightarrow{\gamma}_{(g_0g)}$ the position vector of a point $g$ relative to the origin point $g_0$ in the horizontal sense, which is also equivalent to a horizontal geodesic connecting $g_0$ and  $g$ with an initial tangent direction $\gamma^\prime(g_0)$.
Similarly, we denote by $\nu$ the position vector of a point $\tilde{g}\in\mathcal{S}^{2n}$ relative to the origin point $g_0$ in the horizontal sense.

Let $\Omega\subset\mathbb{H}_{g_0}$ be a $H$-convex body, and  $\nu(g)$ be the horizontal unit direction vector at $g\in\partial\Omega$. We denote $H_{\nu}$ by   the ``horizontal totally geodesic hypersurface" to $\partial\Omega$ at $g$ orthogonal to $\nu$ passing through $g_0$.
It is not hard to check that the ``horizontal totally geodesic hypersurface" to $\partial\Omega$ at $g$ orthogonal to $\nu$ passing through $g_0$
is essentially  equivalent to ``tangent plane" defined below (see e.g., \cite{Mo}).
$$H_{\nu}=\{g=(z,t)\in\partial\Omega: \langle\nu, z\rangle=0,\ t\in\mathbb{R}\}.$$

\begin{remark}\label{r1.0}
If $g$ belongs to the subspace $\mathbb{R}^{2n}\times\{0\}\subset\mathbb{H}_{g_0}$, it follows from the facts of geodesics in \cite{Haj} that the horizontal totally geodesic hypersurface $H_{\nu}$ is exactly  ``an Euclidean tangent plane".
\end{remark}

We now define the horizontal support function (abbreviated $H$-support function) as follows:

\begin{definition}\label{d1.1}
{\it Let $\Omega\subset\mathbb{H}_{g_0}$ be a $H$-convex body with $\emptyset\neq\Omega\neq\mathbb{H}_{g_0}$.
For the horizontal unit direction vector $\nu$,
the $H$-support function, denoted by $h_H(\Omega,\cdot)$, of $\Omega$ is defined as the horizontal geodesic distance from the origin $g_0$ to  $H_\nu$ at $g\in\partial\Omega$.}
\end{definition}
Since the Heisenberg group is isomorphic to Euclidean space, then the $H$-support function can be defined as mapping $h_H(\Omega,\cdot): \mathcal{S}^{2n}\rightarrow\mathbb{R}$, which is further expressed as
\begin{align}\label{1.5}
h_H(\Omega,\nu)=\max\{\langle
\overrightarrow{\gamma}_{(g_0g)},\nu
\rangle_{\mathbb{H}^n}: g\in\Omega\},
\end{align}
for the unit direction vector $\nu$.

Notice that if $g=(0,t)\in\partial\Omega$, $t\neq0$, then $H_\nu$ at $g$ orthogonal to $t$-axis, and there are an  infinite number of the horizontal geodesic connecting $g_0$ and $g$ with initial vector $\gamma^\prime_{g_0}$ just as in (\ref{1.5}) (see Theorem 2.A). In this case, we can do make the following constraints to ensure the uniqueness of the $H$-support function.
For any horizontal unit direction vector $\nu$, and $g=(0,t)\in\partial\Omega$, $t>0$, we can take a plane denoted by $\Pi_{g_0g}$ that crosses the points $g_0$, $g$ and the vector $\nu$ such that $\gamma^\prime_{g_0}$ belonging to $\Pi_{g_0g}$, as shown in Figure 1.
We can now choose the horizontal geodesic $\gamma_{1(g_0g)}$ with an angle of $[0,\frac{\pi}{2}]$ to $\nu$ as the only horizontal geodesic connecting $g_0$ and $g$, and regard $\overrightarrow{\gamma}_{1(g_0g)}$ as the position vector $\overrightarrow{\gamma}_{(g_0g)}$ in (\ref{1.5}). For $t<0$, we can do the similar convention.
\begin{figure}[h!]
\centering
\begin{minipage}{15cm}
\centering
\includegraphics[scale=0.5]{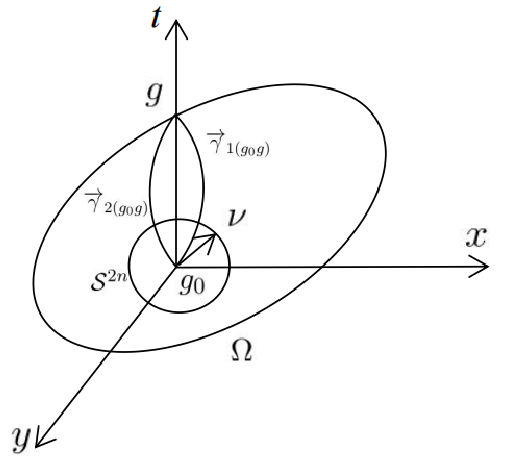}
\end{minipage}
\caption{$H$-support function of $g=(0,t)$.}
\end{figure}

From Definition \ref{d1.1} we can easily verify that $h_H(\Omega,u\ast\nu)\leq h_H(\Omega,u)+h_H(\Omega,\nu)$ and  $h_H(\Omega,\delta_\lambda\nu)=\lambda h_H(\Omega,\nu)$ for $\lambda\in[0,1]$. Hence, $h_H(\Omega,\cdot)$ is a $H$-convex function if $\Omega\neq\mathbb{H}_{g_0}$.
In particular, we get that $h_H(\Omega,\nu)=\langle\overrightarrow{\gamma}_{(g_0g)},
\nu\rangle$ if and only if $\Omega=\{g\}$.
It is also clear from the definition that $h_H(\Omega,\nu)\leq h_H(\tilde{\Omega},\nu)$ if and only if $\Omega\subset\tilde{\Omega}$. Further, we note that $h_H(\delta_\lambda\Omega,\nu)=\lambda h_H(\Omega,\nu)$ for $\lambda\in[0,1]$.

\begin{remark}\label{r1.2}
From the fact of geodesics in \cite{Haj}, if $g$ belongs to the subspace $\mathbb{R}^{2n}\times\{0\}\subset\mathbb{H}_{g_0}$, it is easy to know that $\overrightarrow{\gamma}_{(g_0g)}$ is a unique geodesic and that its length is equal to the Euclidean length $|\overrightarrow{g_0g}|$ of segment $\overrightarrow{g_0g}$. In this case, the $H$-support function is the Euclidean support function.
\end{remark}

A horizontal Gauss map (abbreviated $H$-Gauss map), denoted by $\nu_H$, is defined as follows:
\begin{definition}\label{d1.3}
{\it For $g_0\in\mathbb{H}^n$, let $\Omega\subset\mathbb{H}_{g_0}$ be a $H$-convex body with smooth boundary $\partial\Omega$. The unit outer normal to $\partial\Omega$, denoted by $\nu_H$, is well defined for $\mathcal{H}^{2n}$ almost all point on $\partial\Omega$.
The map $\nu_H: \partial\Omega\rightarrow\mathcal{S}^{2n}$
is called the $H$-Gauss map of $\Omega$.}
\end{definition}

Let $\nu_{H}^{-1}$ be the inverse of the mapping $\nu_H$.
Next, we introduce the horizontal surface area measure (abbreviated $H$-surface area measure).

\begin{definition}\label{d1.3}
{\it For a $H$-convex body $\Omega$ and a Borel measure $\mathfrak{m}$ on $\mathcal{S}^{2n}$, the $H$-surface area measure of $\Omega$ is the pushforward of $\mathfrak{m}$ by $\nu_H$ (i.e., $S_H(\Omega,\cdot)=\nu_{H,\Omega\sharp}^{-1}\mathfrak{m}$)
$$S_H(\Omega,\omega)=\nu_{H,\Omega\sharp}^{-1}\mathfrak{m}(\omega)
=\int_{\nu_{H,\Omega}^{-1}(\omega)}
d\mathcal{H}^{2n},$$
for every Borel set $\omega\subset\mathcal{S}^{2n}$.
Here $\mathcal{H}^{k}$ is the $k$-dimensional Haussdorff measure which depends on the Kor\'{a}nyi distance.}
\end{definition}

Consider a $C^1$ vector field $F$ with compact support on $\mathbb{H}^n$, and denote by $\{\phi_t\}_{|t|<\varepsilon}$, $t\in\mathbb{R}$, the associated group of diffeomorphisms (see \cite{Mo}).
By the variation of $Vol(\Omega_t)$ along any $H$-convex body $\Omega$ at $t=0$, the display expression of the $H$-surface area measure can be derived: (see Section 3 for details):
$$dS_H(\Omega,\omega)
=\det(\nabla_{ij}^Hh_H+h_H\delta_{ij}+A_{ij})da,$$
where $da$ is the area measure of $\mathcal{S}^{2n}$,  $\nabla^H$ is the horizontal gradient of $h_H$ on $\mathcal{S}^{2n}$, $\nabla_{ij}^H=(\nabla_i^H\nabla_j^H+\nabla_j^H\nabla_i^H)/2$ is the horizontal second order covariant derivative, $A_{ij}=\frac{1}{2}\langle \mathcal{G},\nabla_k^Hv\rangle_{\mathbb{H}^n}(
\langle\nabla_i^H\nabla_k^Hv,\nabla_j^Hv\rangle_{\mathbb{H}^n}
+\langle\nabla_j^H\nabla_k^Hv,
\nabla_i^Hv\rangle_{\mathbb{H}^n})$, and $\mathcal{G}$ denotes the inverse $H$-Gauss map.

Now we introduce the following Minkowski type problem with respect to the $H$-surface area measure, which may be called the {\it Heisenberg Minkowksi problem}.

\begin{problem}\label{p1.3}
Given a finite Borel measure $\mu$ on $\mathcal{S}^{2n}$, what are necessary and sufficient conditions for $\mu$ such that there exists a $H$-convex body $\Omega\subset\mathbb{H}_{g_0}$ satisfying
$$dS_H(\Omega,\cdot)=d\mu?$$
\end{problem}

There are generally two methods to solve the Minkowski problem in $\mathbb{R}^n$: variational method and curvature flow method. The existence of smooth solutions can be obtained through the curvature flow, while the variational method can only obtain the existence of weak solutions. Due to the degeneracy of the Laplace equation in the Heisenberg group, the existence of smooth solutions cannot be obtained through the curvature flow method. Therefore, we obtain the existence of solutions to {\bf Problem \ref{p1.3}} through the variational method, see Section 4 details.

\begin{theorem}\label{t1.4}
Let $\mu$ be a finite Borel measure on $\mathcal{S}^{2n}$ that is not concentrated on a closed hemisphere and satisfies
$$\int_{\mathcal{S}^{2n}}ud\mu(u)=0.$$
Then there exists a $H$-convex body $\Omega$ such that
$$S_H(\Omega,\cdot)=\mu.$$
\end{theorem}

The organization of this paper is as follows.
In Section 2, we collect some basic knowledge about the Heisenberg group $\mathbb{H}^n$.
In Section 3, through calculating the variation of
volume by use the associated group of diffeomorphisms,
we introduce the $H$-surface area measure and
Heisenberg Minkowski problem.
In Section 4, we give the existence of solutions to Heisenberg Minkowski problem by the variational method.
In Section 5, we provide the relationship between the $H$-support function and the Euclidean support function, and propose a more general Orlicz Heisenberg Minkowski problem.

\section{Preliminaries}

In this section, we give some basic knowledge about sub-Riemannian Heisenberg group $\mathbb{H}^n$.

\subsection{The equations for geodesics in $\mathbb{H}^n$}

We recall here the equations for geodesics of unit length starting from $e=(0,0,0)$, since all other geodesics can be recovered by left translations and dilations (see e.g., \cite{AF,Cap,Ma,Mon,Y}).
Let $s\in[0,1]$ and $\phi\in[-2\pi,2\pi]$, and let $A_j, B_j\in\mathbb{R}$ such that $\sum_{l=1}^n(A_l^2+B_l^2)=1$,
then the set of equations
\begin{align}\label{2.0}
\begin{cases}
x_l(s)=\frac{A_l(1-\cos(\phi s))+B_l\sin(\phi s)}{\phi}\ \ \ \  l=1,\cdots,n\\
y_l(s)=\frac{-B_l(1-\cos(\phi s))+A_l\sin(\phi s)}{\phi}\ \ \ \  l=1,\cdots,n\\
t(s)=\frac{\phi s-\sin(\phi s)}{2\phi^2}
\end{cases}
\end{align}
defines a geodesic $\gamma(s)$ connecting $(0,0,0)$ with the point $(x_l,y_l,t)$, whose coordinates are
\begin{align*}
\begin{cases}
x_l(s)=x_l(1)=\frac{A_l(1-\cos\phi)+B_l\sin\phi}{\phi}\ \ \ \  l=1,\cdots,n\\
y_l(s)=y_l(1)=\frac{-B_l(1-\cos\phi)+A_l\sin\phi}{\phi}\ \ \ \  l=1,\cdots,n\\
t(s)=t(1)=\frac{\phi-\sin\phi}{2\phi^2}.
\end{cases}
\end{align*}
Of course, this gives a parameterization of the boundary of the CC ball with unit radius.
In the limit case $\phi=0$ one has the Euclidean geodesic.
The existence and uniqueness of geodesics in $\mathbb{H}^n$ has the following results:

\noindent{\bf Theorem 2.A}~~{\it
Let $g=(z,t)\neq e\in\mathbb{H}^n$, we have

i) If $z\neq0$, there exists a unique length minimizing geodesic connecting $e$ and $g$.

ii) If $z=0$, (i.e., $g$ belongs to the center of $\mathbb{H}^n$) then there is a one parameter family of length minimizing geodesics connecting $e$ and $g$, obtained by rotation of a single geodesic around the $t$-axis.}

In addition, there are some facts about geodesics (see e.g., \cite{Haj}): If $g$ belongs to the subspace $\mathbb{R}^{2n}\times\{0\}\subset\mathbb{R}^{2n+1}
=\mathbb{H}^n$, then it is easy to check that the straight line $\gamma(s)=sg$, $s\in[0,1]$ is a unique geodesic (up to a reparametrization) connecting origin $0$ to $g$. Indeed, it is easy to check that $\gamma$ is horizontal, and its length equals the Euclidean length $|\overrightarrow{0g}|$ of the segment $\overrightarrow{0g}$.

\subsection{Hausdorff measure}

For any $m\geq0$ and $\delta>0$, one can  define the premeasures on $\mathbb{H}^n$ associated with Kor\'{a}nyi distance (see \cite{Mo}).
The diameter of a set $K\subset\mathbb{H}^n$ is
$$diam K=\sup_{g,g^\prime\in K}d_{\mathbb{H}}(g,g^{\prime}).$$

(i) $\mathcal{H}_{\mathbb{H}}^m(E)=\lim_{\delta\rightarrow0}
\mathcal{H}_{\mathbb{H}}^{m,\delta}(E)$, where, up to a constant multiple,
$$\mathcal{H}_{\mathbb{H}}^{m,\delta}(E)=
\inf\bigg\{\sum_{i\in\mathbb{N}}(diam(K_i))^m:
E\subset\bigcup_{i\in\mathbb{N}}K_i, diam(K_i)<\delta\bigg\},$$
and the infimum is taken with respect to any non-empty family of closed subsets $\{K_i\}_{i\in\mathbb{N}}\subset\mathbb{H}^n$;

(ii) $\mathcal{S}_{\mathbb{H}}^m(E)=\lim_{\delta\rightarrow0}
\mathcal{S}_{\mathbb{H}}^{m,\delta}(E)$, where, up to a constant multiple,
$$\mathcal{S}_{\mathbb{H}}^{m,\delta}(E)=
\inf\bigg\{\sum_{i\in\mathbb{N}}(diam(B_i))^m:
E\subset\bigcup_{i\in\mathbb{N}}B_i, diam(B_i)<\delta\bigg\},$$
and the infimum is taken with respect to closed Kor\'{a}nyi balls $B_i$. The measures $\mathcal{H}_{\mathbb{H}}^m$ and $\mathcal{S}_{\mathbb{H}}^m$ are called, respectively, $m$-dimensional horizontal Hausdorff measure and $m$-dimensional horizontal spherical Hausdorff measure.

\subsection{Horizontal divergence and twisted convex combination}

Let $V$ be a smooth vector field in $\mathbb{H}^n=\mathbb{R}^{2n+1}$. We may express $V$ using both the basis $\{X_l, Y_l, T\}$ ($l=1,\cdots,n$) and the standard basis of vector fields of $\mathbb{R}^{2n+1}$:

\begin{align*}
V&=\sum_{l=1}^n(\varphi_lX_l+\varphi_{n+l}Y_l)+\varphi_{2n+1}T\\
&=\sum_{l=1}^n\bigg[\varphi_l\frac{\partial}{\partial x_l}
+\varphi_{n+l}\frac{\partial}{\partial y_l}
+(2y_l\varphi_l-2x_l\varphi_{n+l})\frac{\partial}{\partial t}\bigg]+\varphi_{2n+1}\frac{\partial}{\partial t},
\end{align*}
where $\varphi_l, \varphi_{n+l}, \varphi_{2n+1}\in C^{\infty}(\mathbb{H}^n)$ are smooth functions.
The standard divergence of $V$ is
\begin{align*}
div V&=\sum_{l=1}^n\bigg[\frac{\partial\varphi_l}{\partial x_l}+\frac{\partial\varphi_{n+l}}{\partial y_l}
+2y_l\frac{\partial\varphi_l}{\partial t}-2x_l\frac{\partial\varphi_{n+l}}{\partial t}\bigg]
+\frac{\partial\varphi_{2n+1}}{\partial t}\\
&=\sum_{l=1}^n(X_l\varphi_l+Y_l\varphi_{n+l})+T\varphi_{2n+1}.
\end{align*}
The vector field $V$ is said to be horizontal if $V(g)\in\mathbb{H}_g$ for all $g\in\mathbb{H}^n$.
These observations suggest the following definition.

Let $A\subset\mathbb{H}^n$ be an open set. The horizontal divergence of a vector valued mapping $\varphi\in C^1(A; \mathbb{R}^{2n})$ is defined by
\begin{align}\label{8}
div_H \varphi=\sum_{l=1}^n(X_l\varphi_l+Y_l\varphi_{n+l}).
\end{align}
Thus $div_H \varphi=div V$ is the standard divergence of the horizontal vector field $V$ with coordinates $\varphi=(\varphi_1,\cdots, \varphi_{2n})$ in the basis $\{X_l, Y_l\}$.

Let $g=(x,y,t)$ and $g^\prime=(x^\prime,y^\prime,t^\prime)$ be two points in $\mathbb{H}^n$, and $\lambda\in[0,1]$. The twisted convex combination is defined by (see \cite{D})
\begin{align}\label{2.1}
g\ast\delta_\lambda(g^{-1}\ast g^\prime)
&=(x+\lambda(x^\prime-x),y+\lambda(y^\prime-y),
t+2\lambda(xy^\prime-x^\prime y)\\
&\nonumber\ \ +\lambda^2(t^\prime-t+2x^\prime y-2xy^\prime)).
\end{align}
If $g^\prime\in\mathbb{H}_g$, i.e. $g^{-1}\ast g^\prime\in\mathbb{H}_e$, then we get
\begin{align}\label{2.2}
t^\prime-t+2(x^\prime y-xy^\prime)=0.
\end{align}
From (\ref{2.1}) and (\ref{2.2}), we derive for any $g^\prime\in\mathbb{H}_g$ that
\begin{align}\label{2.3}
g\ast\delta_\lambda(g^{-1}\ast g^\prime)=(1-\lambda)g+\lambda g^\prime.
\end{align}
Note also that (\ref{2.2}) and (\ref{2.3}) hold and $g^\prime\in\mathbb{H}_g$ if and only if $g^{-1}\ast g^\prime\in\mathbb{H}_e$.

\

\section{$H$-surface area measure and Heisenberg Minkowski problem}

In this section, we first calculate the second fundamental form of $H$-support function.
Then, we introduce the $H$-surface area measure and Heisenberg Minkowski problem in $\mathbb{H}_{g_0}$ by calculating the variation of volume of the associated group of diffeomorphisms.
Before that, we provide a fact about the covariant derivative of the Riemannian connection of the
left-invariant metric.

\begin{lemma}\label{l3.1}(\cite{Ma})
For the covariant derivatives of the Riemannian connection of the left-invariant metric, the following is true:
\begin{equation*}
\nabla=\left(
\begin{array}{ccccc}
0&T&\ \ &\cdots&-Y_1\\
-T&0&\ \ &\cdots&X_1\\
\vdots&\ \ &\ddots&\ \ &\vdots\\
0&\cdots&0&T&-Y_n\\
0&\cdots&-T&0&X_n\\
-Y_1&X_1&\ \ &\cdots&0
\end{array}
\right),
\end{equation*}
where the $(i,j)$-element in the table above equals $\nabla_{e_i}e_j$ for our basis
$$\{e_l,l=1,\cdots,2n+1\}=\{X_1,Y_1,X_2,Y_2,\cdots,T\}.$$
\end{lemma}

\begin{theorem}\label{t3.1}
Fix a point $g_0\in\mathbb{H}^n$, let $\Omega\subset\mathbb{H}_{g_0}$ be a $H$-convex body with smooth boundary $\partial\Omega$. Relative to the orthonormal frame $\{e_l,l=1,\cdots,2n\}$, the second fundamental form of $\partial\Omega$ is
$$\Pi_{ij}^H=\nabla_{ij}^Hh_H+h_H\delta_{ij}+A_{ij},$$
where $\nabla_{ij}^H\hat{=}\frac{1}{2}(\nabla_i^H\nabla_j^H+\nabla_j^H\nabla_i^H)$,
$A_{ij}=\frac{1}{2}\langle \mathcal{G},\nabla_k^Hv\rangle_{\mathbb{H}^n}(
\langle\nabla_i^H\nabla_k^Hv,\nabla_j^Hv\rangle_{\mathbb{H}^n}
+\langle\nabla_j^H\nabla_k^Hv,
\nabla_i^Hv\rangle_{\mathbb{H}^n})$, and $\mathcal{G}$ denotes the inverse $H$-Gauss map. Moreover, if $\Omega$ is a convex body in Euclidean space, then $A_{ij}=0$. In this case, the horizontal second fundamental form is the Euclidean second fundamental form (see e.g., \cite{U}).
\end{theorem}

\begin{proof}
Since $\partial\Omega$ is a smooth, closed and $H$-convex hypersurface, we may assume that $\partial\Omega$ is parametrized by inverse $H$-Gauss map. Without loss of generality, we assume that $\partial\Omega$ encloses the origin $g_0$. From (\ref{1.5}), the $H$-support function of $\partial\Omega$ can be defined, for $v\in\mathcal{S}^{2n}$, as
\begin{align}\label{3.1}
h_H(\partial\Omega,v)=\langle\mathcal{G}(v),
v\rangle_{\mathbb{H}^n}.
\end{align}

Let $\{e_l,l=1,\cdots,2n\}$ be a smooth orthonormal frame field on $\mathcal{S}^{2n}$, and let $\nabla^H$ be the horizontal gradient on $\mathcal{S}^{2n}$.
Differentiating (\ref{3.1}) we obtain
$$\nabla_i^Hh_H=\langle\nabla_i^H\mathcal{G},
v\rangle_{\mathbb{H}^n}
+\langle\mathcal{G},\nabla_i^Hv\rangle_{\mathbb{H}^n}.$$
Since $\nabla_i^H\mathcal{G}$ is tangential and $v$ is the normal, one has
$$\nabla_i^Hh_H=\langle\mathcal{G},
\nabla_i^Hv\rangle_{\mathbb{H}^n}.$$
Differentiating once again, and writing $(\nabla_i^H\nabla_j^H+\nabla_j^H\nabla_i^H)/2=\nabla_{ij}^H$, we obtain
\begin{align}\label{3.1.1}
\nonumber\nabla_{ij}^Hh_H&=\frac{1}{2}(\nabla_i^H\nabla_j^Hh_H
+\nabla_j^H\nabla_i^Hh_H)\\
\nonumber&=\frac{1}{2}(\langle\mathcal{G},
\nabla_i^H\nabla_j^Hv\rangle_{\mathbb{H}^n}
+\langle\mathcal{G},
\nabla_j^H\nabla_i^Hv\rangle_{\mathbb{H}^n}\\
&\ \ \ \ +\langle\nabla_i^H\mathcal{G},
\nabla_j^Hv\rangle_{\mathbb{H}^n}+\langle\nabla_j^H\mathcal{G},
\nabla_i^Hv\rangle_{\mathbb{H}^n})\\
\nonumber&=\langle\mathcal{G},
\frac{1}{2}(\nabla_i^H\nabla_j^Hv
+\nabla_j^H\nabla_i^Hv)\rangle_{\mathbb{H}^n}
+\Pi_{ij}^H\\
\nonumber&=\langle\mathcal{G},
\nabla_{ij}^Hv\rangle_{\mathbb{H}^n}+\Pi_{ij}^H,
\end{align}
where $\Pi_{ij}^H$ is the horizontal second fundamental form of $\partial\Omega$. To compute the term $\langle \mathcal{G},\nabla_{ij}^Hv\rangle_{\mathbb{H}^n}$ we differentiate the equation $\langle v,v\rangle_{\mathbb{H}^n}=1$ and obtain
\begin{align}\label{3.2}
\langle v,\nabla_i^Hv\rangle_{\mathbb{H}^n}=0,
\end{align}
and
\begin{align*}
\langle v,\nabla_i^H\nabla_j^Hv\rangle_{\mathbb{H}^n}
+\langle\nabla_{i}^Hv,\nabla_{j}^Hv\rangle_{\mathbb{H}^n}=0.
\end{align*}
Since $\{e_l,l=1,\cdots,2n\}$ is an orthonormal frame field defined in Lemma \ref{l3.1}, by a direct computation, we arrive at
\begin{align}\label{3.3}
\langle v,\nabla_i^H\nabla_j^Hv\rangle_{\mathbb{H}^n}
=-\langle\nabla_{i}^Hv,\nabla_{j}^Hv\rangle_{\mathbb{H}^n}
\ \hat{=}-\delta_{ij}.
\end{align}
Finally
\begin{align}\label{3.4}
\langle\nabla_k^H\nabla_i^Hv,
\nabla_{j}^Hv\rangle_{\mathbb{H}^n}
+\langle\nabla_k^H\nabla_j^Hv,
\nabla_{i}^Hv\rangle_{\mathbb{H}^n}=0.
\end{align}

From (\ref{3.2}) and (\ref{3.3}), we see that $\nabla_1^Hv,\cdots,\nabla_{2n}^Hv$ form an orthonormal basis for horizontal tangent space of $\partial\Omega$,
and hence
\begin{align*}
\langle\mathcal{G},\nabla_{ij}^Hv\rangle_{\mathbb{H}^n}
=\langle\mathcal{G},\frac{1}{2}(
\nabla_i^H\nabla_j^Hv
+\nabla_j^H\nabla_i^Hv)\rangle_{\mathbb{H}^n},
\end{align*}
where
\begin{align*}
\langle\mathcal{G},\nabla_i^H\nabla_j^Hv\rangle_{\mathbb{H}^n}
&=\langle\langle\mathcal{G},v\rangle_{\mathbb{H}^n}v,
\nabla_i^H\nabla_j^Hv\rangle_{\mathbb{H}^n}+\langle\langle \mathcal{G},\nabla_k^Hv\rangle_{\mathbb{H}^n}\nabla_k^Hv,
\nabla_i^H\nabla_j^Hv\rangle_{\mathbb{H}^n}\\
&=\langle\mathcal{G},v\rangle_{\mathbb{H}^n}\langle v,
\nabla_i^H\nabla_j^Hv\rangle_{\mathbb{H}^n}
+\nabla_i^H\langle\langle \mathcal{G},\nabla_k^Hv\rangle_{\mathbb{H}^n}\nabla_k^Hv,
\nabla_j^Hv\rangle_{\mathbb{H}^n}\\
&\ \ -\langle\nabla_i^H(\langle \mathcal{G},\nabla_k^Hv\rangle_{\mathbb{H}^n}\nabla_k^Hv),
\nabla_j^Hv\rangle_{\mathbb{H}^n}\\
&=-h_H\delta_{ij}+\nabla_i^H\langle \mathcal{G},\nabla_j^Hv\rangle_{\mathbb{H}^n}
-\nabla_i^H\langle\mathcal{G},\nabla_k^Hv\rangle_{\mathbb{H}^n}
\langle\nabla_k^Hv,\nabla_j^Hv\rangle_{\mathbb{H}^n}\\
&\ \ -\langle\mathcal{G},\nabla_k^Hv\rangle_{\mathbb{H}^n}
\langle\nabla_i^H\nabla_k^Hv,\nabla_j^Hv\rangle_{\mathbb{H}^n}\\
&=-h_H\delta_{ij}-\langle \mathcal{G},\nabla_k^Hv\rangle_{\mathbb{H}^n}
\langle\nabla_i^H\nabla_k^Hv,\nabla_j^Hv\rangle_{\mathbb{H}^n},
\end{align*}
and
\begin{align*}
\langle\mathcal{G},\nabla_j^H\nabla_i^Hv\rangle_{\mathbb{H}^n}
=-h_H\delta_{ij}-\langle \mathcal{G},\nabla_k^Hv\rangle_{\mathbb{H}^n}
\langle\nabla_j^H\nabla_k^Hv,\nabla_i^Hv\rangle_{\mathbb{H}^n}
\end{align*}
by virtue of (\ref{3.2}) and (\ref{3.3}). Thus we get
\begin{align}\label{3.5}
\langle\mathcal{G},\nabla_{ij}^Hv\rangle_{\mathbb{H}^n}
&=-h_H\delta_{ij}-A_{ij},
\end{align}
where $A_{ij}=\frac{1}{2}\langle \mathcal{G},\nabla_k^Hv\rangle_{\mathbb{H}^n}(
\langle\nabla_i^H\nabla_k^Hv,\nabla_j^Hv\rangle_{\mathbb{H}^n}
+\langle\nabla_j^H\nabla_k^Hv,
\nabla_i^Hv\rangle_{\mathbb{H}^n})$.

Substituting (\ref{3.5}) into (\ref{3.1.1}), we have
\begin{align*}
\Pi_{ij}^H=\nabla_{ij}^Hh_H+h_H\delta_{ij}+A_{ij}.
\end{align*}

Moreover, if $\Omega$ is a convex body in Euclidean space, then, for the outer unit normal $v$, $\nabla_i^H\nabla_j^Hv=\nabla_j^H\nabla_i^Hv$. Using (\ref{3.4}), we have
$$A_{ij}=\frac{1}{2}\langle \mathcal{G},\nabla_k^Hv\rangle_{\mathbb{H}^n}(
\langle\nabla_k^H\nabla_i^Hv,\nabla_j^Hv\rangle_{\mathbb{H}^n}
+\langle\nabla_k^H\nabla_j^Hv,
\nabla_i^Hv\rangle_{\mathbb{H}^n})
=0.$$
In this case, the horizontal second fundamental form is the Euclidean second fundamental form.
This ends the proof of Theorem \ref{t3.1}
\end{proof}

From Theorem \ref{t3.1}, we can compute the metric $g_{ij}$ of $\partial\Omega$.
By the Gauss-Weingarten relations
$$\nabla_i^Hv=\Pi_{ik}^Hg^{kl}\nabla_l^H\mathcal{G},$$
from which we obtain
$$\delta_{ij}=\langle\nabla_i^Hv,
\nabla_j^Hv\rangle_{\mathbb{H}^n}
=\Pi_{ik}^Hg^{kl}\Pi_{jm}^Hg^{ms}
\langle\nabla_l^H\mathcal{G},\nabla_s^H\mathcal{G}
\rangle_{\mathbb{H}^n}
=\Pi_{ik}^H\Pi_{jl}^Hg^{kl}.$$
Since $\Omega$ is $H$-convex, $\Pi_{ij}^H$ is invertible, and hence
\begin{align}\label{3.6}
g_{ij}=\Pi_{ik}^H\Pi_{jk}^H.
\end{align}

Next, through computing the variation of volume by use the associated group of diffeomorphisms, we introduce the of $H$-surface area measure.

\begin{theorem}\label{t3.2}
For a point $g_0\in\mathbb{H}^n$, let $\Omega\subset\mathbb{H}_{g_0}$ be a $H$-convex body, then
$$Vol^\prime(0)=\int_{\mathcal{S}^{2n}}h_H(\Omega,v)
\det(\nabla_{ij}^Hh_H+h_H\delta_{ij}+A_{ij})da,$$
where $da$ is the area element of $\mathcal{S}^{2n}$.
\end{theorem}

\begin{proof}
For any $\varphi\in C_c^1(\Omega; \mathbb{R}^{2n})$, let $V=\sum_{l=1}^n(\varphi_lX_l+\varphi_{n+l}Y_l)$ be the horizontal vector field with coordinates $\varphi=(\varphi_1,\cdots, \varphi_{2n})$.
By the standard divergence theorem, we have
\begin{align*}
Vol(0)^\prime=\int_\Omega\frac{d}{dt}
\bigg|_{t=0}(d(\sigma_H^{n})_t)
=\int_\Omega div Vd\sigma_H^{n}
=\int_{\partial\Omega}
\langle V,v\rangle_{\mathbb{H}^n} d\sigma_H^{n-1},
\end{align*}
where $d\sigma_H^{n}$ stands for the volume element of $\Omega$, and $d\sigma_H^{n-1}$ is the area element of $\partial\Omega$ at $\mathcal{G}$.
In the second equality we have used that  $\frac{d}{dt}|_{t=0}(d(\sigma_H^{n})_t)
=div Vd\sigma_H^{n}$ (see e.g., \cite{Fr,Sim}).

Let $da$ be the area element of $\mathcal{S}^{2n}$ at $v$.
By (\ref{3.6}), the area element $d\sigma_H^{n-1}$ of $\partial\Omega$ at $\mathcal{G}$ can be given by
\begin{align*}
d\sigma_H^{n-1}&=\sqrt{\det g_{ij}}da
=\sqrt{\det(\Pi_{ij}^H)^2}da\\
&=\det(\nabla_{ij}^Hh_H+h_H\delta_{ij}+A_{ij})da.
\end{align*}
This, together with the definition of $H$-support function, yields
\begin{align*}
Vol(0)^\prime=\int_{\mathcal{S}^{2n}}
h_H(\Omega,v)\det(\nabla_{ij}^Hh_H+h_H\delta_{ij}+A_{ij})da.
\end{align*}
This completes the proof of Theorem \ref{t3.2}
\end{proof}

\begin{definition}\label{d3.1}
{\it For a point $g_0\in\mathbb{H}^n$, let $\Omega\subset\mathbb{H}_{g_0}$ be a $H$-convex body. The $H$-surface area measure $S_H(\Omega,\cdot)$ of $\Omega$ can be defined by
$$dS_H(\Omega,\cdot)=\det(\nabla_{ij}^Hh_H
+h_H\delta_{ij}+A_{ij})da.$$}
\end{definition}

\begin{remark}\label{r3.2}
If $\Omega$ is a convex body in Euclidean space, then the $H$-surface area measure $S_H(\Omega,\cdot)$ is the Euclidean surface area measure (see \cite{S} or \cite{U}).
\end{remark}

By Theorem \ref{t3.2}, Definitions \ref{d1.3} and \ref{d3.1} and horizontal divergence theorem (see e.g., \cite{Mon1}), we have
\begin{align*}
Vol(\Omega)&=\int_\Omega d\sigma_H^n
=\frac{1}{2n+1}\int_\Omega div Vd\sigma_H^n\\
&=\frac{1}{2n+1}\int_{\partial\Omega}\langle V,\nu\rangle_{\mathbb{H}^n}d\mathcal{H}^{2n}\\
&=\frac{1}{2n+1}\int_{\mathcal{S}^{2n}}
h_H(\Omega,\nu)dS_H(\Omega,\nu).
\end{align*}

This motivates us to propose the following Heisenberg Minkowski problem for Definition \ref{d3.1}.

\begin{problem}\label{p3.1}
Given a finite Borel measure $\mu$ on $\mathcal{S}^{2n}$, what are necessary and sufficient conditions for $\mu$ such that there exists a $H$-convex body $\Omega\subset\mathbb{H}_{g_0}$ satisfying
$$dS_H(\Omega,\cdot)=d\mu?$$
\end{problem}

\

\section{Existence of solutions to Heisenberg Minkowski problem}

In this section, we study the existence of solutions to Heisenberg Minkowski problem. The main methods used to solve this problem is the so-called variational method. First we need to introduce the concept of Wulff shape.

Let $\omega$ be a closed subset of $\mathcal{S}^{2n}$ that is not contained on a closed hemisphere, and let $f: \mathcal{S}^{2n}\rightarrow\mathbb{R}$ be a positive continuous function. The closed convex set
$$\Omega_f=\bigcap_{u\in\omega}H^-(u,f(u)),$$
is bounded, since $\omega$ positively spans $\mathbb{H}_{g_0}$. Here
$$H^-(u,f(u))=\{g\in\mathbb{H}_{g_0}:
\langle\overrightarrow{\gamma}_{(g_0g)},u
\rangle_{\mathbb{H}^n}\leq f(u)
\ for \ all\ u\in\omega\}.$$
It is a $H$-convex body that contains the origin in its interior, since the restriction of $f$ to $\omega$ has a positive lower bound. The body $\Omega_f$ is called {\it the horizontal Wulff shape} associated with $f(u)$.

The Wulff shape can be also defined as the unique element of
$$\{\Omega\in\mathcal{K}_0^H: h_H(\Omega,u)\leq f(u), \ u\in\omega\},$$
where $\mathcal{K}_0^H$ denotes the set of $H$-convex body that contain the origin in their interior.

Now we prove the existence of solutions to Heisenberg Minkowski problem.

\begin{theorem}\label{t5.1}
For $g_0\in\mathbb{H}^n$ and the Kor\'{a}nyi unit sphere $\mathcal{S}^{2n}$,
let $\mu$ be a finite Borel measure on $\mathcal{S}^{2n}$ that is not concentrated on a closed hemisphere and satisfies
$$\int_{\mathcal{S}^{2n}}ud\mu(u)=0.$$
Then there exists a $H$-convex body $\Omega$ in $\mathbb{H}_{g_0}$ such that
$$S_H(\Omega,\cdot)=\mu.$$
\end{theorem}

\begin{proof}
Let
$$\mathcal{R}_{\Omega}=\max_{u\in\mathcal{S}^{2n}}
h_H(\Omega,u),$$
and
$$\|f\|=\int_{\mathcal{S}^{2n}}f(u)d\mu(u)$$
for a positive continuous function $f$ on $\mathcal{S}^{2n}$.

Let $C^+(\mathcal{S}^{2n})$ be the set of positive continuous functions on $\mathcal{S}^{2n}$.
Define the continuous functional
$$\Phi: C^+(\mathcal{S}^{2n})\rightarrow(0,\infty)$$
by
$$\Phi(f)=\frac{\|f\|}{Vol(\Omega_f)^{\frac{1}{2n+1}}},$$
where $\Omega_f$ is the Wulff shape of $f$.
It is easy to verify that $\Phi$ is homogeneous of degree $0$ and monotonically increasing.

Next, we consider the following minimization problem
$$\inf_{f\in C^+(\mathcal{S}^{2n})}\Phi(f).$$

We are searching for a function at which $\Phi$ attains a minimum. For any $f\in C^+(\mathcal{S}^{2n})$, from the definition of Wulff shape, there is $h_H(\Omega_f,\cdot)\leq f$. It implies that $\Phi(h_H(\Omega_f,\cdot))\leq\Phi(f)$.
Therefore, one can search for the minimum point of $\Phi$ among the $H$-support function of $H$-convex bodies that contain the origin in their interior. Since $\Phi$ is homogeneous of degree $0$, the minimum of $\Phi$ is
$$\inf\{\Phi(f): f\in C^+(\mathcal{S}^{2n})\}
=\{\|h_H(\Omega,\cdot)\|: Vol(\Omega)=\omega_{2n}\},$$
where $\omega_{2n}$ denotes the volume of Kor\'{a}nyi unit ball.

Let $\Omega_j$ be a minimizing sequence for $\Phi$, that is
\begin{align}\label{4.1}
\lim_{j\rightarrow\infty}\Phi(h_H(\Omega_j,\cdot))
=\{\|h_H(\Omega,\cdot)\|: Vol(\Omega)=\omega_{2n}\}.
\end{align}
The $H$-convex bodies $\Omega_j$ contain the origin in their interior.

We claim that the sequence $\Omega_j$ is bounded.

Since $\Omega_j$ is a minimizing sequence, we have $\Phi(h_H(\Omega_j,\cdot))<\Phi(1)+1$ when $j$ is large.
Let $\mathcal{R}_{\Omega_j}$ be the maximal radius of $\Omega_j$. If $u_j$ is the direction of this radius, then
$$\mathcal{R}_{\Omega_j}\langle u_j, v\rangle^+_{\mathbb{H}^n}\leq h_H(\Omega_j,v),$$
for all $v\in\mathcal{S}^{2n}$. We have
\begin{align*}
\mathcal{R}_{\Omega_j}\int_{\mathcal{S}^{2n}}\langle u_j,v\rangle^+_{\mathbb{H}^n}d\mu(v)&\leq\int_{\mathcal{S}^{2n}}
h_H(\Omega_j,v)d\mu(v)\\
&\leq\mu(\mathcal{S}^{2n})+1.
\end{align*}
Since $\mu$ is not concentrated on a closed hemisphere, there exists a constant $c>0$ such that
$$\int_{\mathcal{S}^{2n}}\langle u_j, v\rangle^+_{\mathbb{H}^n}d\mu(u)\geq c$$
for all $v\in\mathcal{S}^{2n}$.
Thus
$$\mathcal{R}_{\Omega_j}
\leq\frac{\mu(\mathcal{S}^{2n})+1}{c}.$$
Hence, the sequence $\Omega_j$ is bounded.

By the Blaschke selection theorem, there exists a subsequence of $\Omega_j$ which converges to a $H$-convex body $\Omega_0$ that contains the origin. Since $Vol(\Omega_j)=\omega_{2n}$ in (\ref{4.1}), there is $Vol(\Omega_0)=\omega_{2n}$.
Since the horizontal Wulff shape associated with the $H$-support function $h_H(\Omega,\cdot)$ of a $H$-convex body $\Omega$ is the body $\Omega$ itself, it follows from the assumption of $\Phi$ that
$$\Phi(h_H(\Omega,\cdot))=\frac{\|h_H(\Omega,\cdot)\|}
{Vol(\Omega)^{\frac{1}{2n+1}}}.$$
By the left translation invariance of Kor\'{a}nyi gauge, we have
$$\overrightarrow{\gamma}_{[(\tilde{g}\ast g_0)(\tilde{g}\ast g)]}
=\overrightarrow{\gamma}_{(g_0g)},$$
this implies that $h_H(\Omega,\cdot)$ is left translation invariant. Thus $\Phi(h_H(\Omega,\cdot))$ is translation invariant via the left translation invariance of $V(\Omega)$. Therefore, any left translation of $\Omega_0$ is still a minimal point of $\Phi$. By this, one can assume that $\Omega_0$ contains the origin in its interior. Thus,
$h_H(\Omega_0,u)>0$ for all $u\in\mathcal{S}^{2n}$.

Consider a deformation of $h_H(\Omega_0,u)$, for any $f\in C^+(\mathcal{S}^{2n})$,
$$h(u,t)=h_H(\Omega_0,u)+tf(u),$$
which is positive when $t$ is small. Since $h_H(\Omega_0,\cdot)$ is a minimal point of $\Phi$, we have
$$\frac{d}{dt}\Phi(h(u,t))\bigg|_{t=0}=0.$$
This equation can be written as
$$-\frac{\|h_H(\Omega_0,u)\|}{(2n+1)Vol(h_H(\Omega_0,u))}
Vol^\prime(h_H(\Omega_0,u))
+\int_{\mathcal{S}^{2n}}f(u)d\mu(u)=0.$$
By the variational formula in Theorem \ref{t3.2}
$$Vol^\prime(h_H(\Omega_0,u))
=\int_{\mathcal{S}^{2n}}f(u)dS_H(\Omega_0,u).$$
It follows that
$$\frac{\|h_H(\Omega_0,u)\|}{(2n+1)Vol(h_H(\Omega_0,u))}
\int_{\mathcal{S}^{2n}}f(u)dS_H(\Omega_0,u)
=\int_{\mathcal{S}^{2n}}f(u)d\mu(u)$$
for any $f\in C^+(\mathcal{S}^{2n})$. This gives that
$$\frac{\|h_H(\Omega_0,u)\|}{(2n+1)Vol(h_H(\Omega_0,u))}
S_H(\Omega_0,u)=\mu.$$
Let $\Omega$ be the dilation of $\Omega_0$ satisfiying
$$S_H(\Omega,\cdot)
=\frac{\|h_H(\Omega_0,u)\|}{(2n+1)Vol(h_H(\Omega_0,u))}
S_H(\Omega_0,\cdot).$$
Therefore, we obtain
$$S_H(\Omega,\cdot)=\mu.$$
\end{proof}

\

In the following Remark, we propose a more general Minkowski-type problem, the Orlicz Heisenberg Minkowski problem, in Heisenberg groups,  which will be investigated in our next article.

\begin{remark}\label{r4.1}
 Inspired by the more general Orlicz case of the classical Minkowski problem in Euclidean spaces,
we can naturally propose the following Orlicz form of Heisenberg Minkowski problem, which may be called the {\it Orlicz Heisenberg Minkowski problem} and stated as below

\begin{problem}\label{p4.1} (Orlicz Heisenberg Minkowski problem)
What are the necessary and sufficient conditions on a given function $\phi$ and a given finite Borel measure $\mu$ on $\mathcal{S}^{2n}$, such that there exists a $H$-convex body $\Omega$, satisfying
$$c\phi(h_H(\Omega,\cdot))dS_H(\Omega,\cdot)=d\mu,$$
where $c>0$ is a constant?
\end{problem}

\end{remark}

\end{document}